\documentclass{amsart}
\usepackage{amsfonts,amssymb,amsmath,amsthm,mathrsfs}
\usepackage{url}
\usepackage{enumerate}

\newtheorem{theorem}{Theorem}[section]
\newtheorem{lemma}[theorem]{Lemma}

\newtheorem{corollary}[theorem]{Corollary}
\theoremstyle{definition}
\newtheorem{definition}[theorem]{Definition}
\newtheorem{remark}[theorem]{Remark}
\newtheorem{assumption}[theorem]{Assumption}
\numberwithin{equation}{section}

\author[G. Hu]{Guoen Hu}
\address{Guoen Hu, School of Applied Mathematics, Beijing Normal University, Zhuhai 519087,
P. R. China}
\email{huguoen@yahoo.com}
\author[Y. Zhang]{Yandan Zhang}
\address{Yandan Zhang, School of Mathematics and Physics, Qingdao University of Science and Technology, Qingdao 266061
P. R. China}
\email{yandzhang@163.com}
\thanks{The  research of the second author was supported by the Natural Science Foundation of Shandong Province (Grant No. ZR2016AB07)}
\begin{document}

\title[singular integral]{The composition of  singular integral operators with nonsmooth kernels}

\begin{abstract}Let $T_1$, $T_2$ be two singular integral operators with nonsmooth kernels  introduced by Duong and McIntosh.
In this paper, by establishing certain bi-sublinear sparse dominations, the authors
obtain some quantitative   bounds on $L^p(\mathbb{R}^n,\,w)$ with $p\in(1,\,\infty)$ and $w\in A_p(\mathbb{R}^n)$ for  the composite operator  $T_1T_2$. Some weighted weak type endpoint estimates are also given.\end{abstract}
\maketitle
\section{Introduction}
We will work on $\mathbb{R}^n$, $n\geq 1$. Let $T$ be an $L^2(\mathbb{R}^n)$ bounded linear operator with kernel $K$ in the sense that
for all $f\in L^2(\mathbb{R}^n)$ with compact support and a. e. $x\in\mathbb{R}^n\backslash {\rm supp}\, f$,
\begin{eqnarray}\label{equ:1.1}Tf(x)=\int_{\mathbb{R}^n}K(x,\,y)f(y)dy,\end{eqnarray}
where $K$ is a measurable function on $\mathbb{R}^n\times \mathbb{R}^n\backslash\{(x,\,y):\,x=y\}$. To obtain a weak
$(1,\,1)$ estimate for certain Riesz transforms, and $L^p$ boundedness with $p\in (1,\,\infty)$ of
holomorphic functional calculi of linear elliptic operators on irregular domains, Duong
and McIntosh \cite{duongmc} introduced singular integral operators with nonsmooth kernels  via the following generalized approximation to the
identity.

\begin{definition}\label{defn1.1}
A family of operators $\{A_t\}_{t>0}$ is said to be an approximation to the identity, if for
every $t>0$, $A_t$ can be represented by the kernel $a_t$ in the following sense: for every function $u\in L^p(\mathbb{R}^n)$
with $p\in [1,\,\infty]$ and a. e.  $x\in\mathbb{R}^n$,
$$A_tu(x)=\int_{\mathbb{R}^n}a_t(x,\,y)u(y)dy,$$
and the kernel $a_t$ satisfies that for all $x,\,y\in\mathbb{R}^n$ and $t>0$,
\begin{eqnarray}\label{equ:1.2}|a_t(x,\,y)|\le h_t(x,\,y)=t^{-n/s}h\Big(\frac{|x-y|}{t^{1/s}}\Big),\end{eqnarray}
where $s>0$ is a constant and $h$ is a positive, bounded and decreasing function such that for some constant $\eta>0$,
\begin{eqnarray}\label{equ:1.3}\lim_{r\rightarrow\infty}r^{n+\eta}h(r)=0.\end{eqnarray}
\end{definition}
\begin{assumption}\label{a1.0}
There exists an approximation to the identity $\{A_t\}_{t>0}$ such that the composite
operator $TA_t$  has an associated kernel $K_t$ in the sense of (\ref{equ:1.1}), and there
exists a positive constant $c_1$  such that for all $y\in \mathbb{R}^n$ and $t>0$,
$$\int_{|x-y|\geq c_1t^{\frac{1}{s}}}|K(x,\,y)-K_t(x,\,y)|dx\lesssim 1.$$\end{assumption}
An $L^2(\mathbb{R}^n)$ bounded linear operator with kernel $K$ satisfying Assumption \ref{a1.0} is called a singular
integral operator with nonsmooth kernel, since $K$ does not enjoy smoothness in space
variables. Duong and McIntosh \cite{duongmc} proved that if $T$ is an $L^2(\mathbb{R}^n)$ bounded linear operator with
kernel $K$, and satisfies Assumption \ref{a1.0}, then $T$ is bounded from $L^1(\mathbb{R}^n)$ to $L^{1,\,\infty}(\mathbb{R}^n)$.
Let $A_p(\mathbb{R}^n)$ ($p\in [1,\,\infty)$) be the weight functions class of Muckenhoupt, that is, $w\in A_{p}(\mathbb{R}^n)$ if $w$ is nonnegative and locally integrable, satisfies that
$$[w]_{A_p}:=\sup_{Q}\Big(\frac{1}{|Q|}\int_Qw(x)dx\Big)\Big(\frac{1}{|Q|}\int_{Q}w^{-\frac{1}{p-1}}(x)dx\Big)^{p-1}<\infty
$$
if $p\in(1,\,\infty)$, where the  supremum is taken over all cubes in $\mathbb{R}^n$,   and
$$[w]_{A_1}:=\sup_{x\in\mathbb{R}^n}\frac{Mw(x)}{w(x)},$$
with $M$   the Hardy-Littlewood maximal operator. $[w]_{A_p}$ is called the $A_p$ constant of $w$, see \cite{gra} for properties of $A_p(\mathbb{R}^n)$. For a weight $u\in A_{\infty}(\mathbb{R}^n)=\cup_{p\geq 1}A_p(\mathbb{R}^n)$, we define $[u]_{A_{\infty}}$,  the $A_{\infty}$ constant of $u$,  by
$$[u]_{A_{\infty}}=\sup_{Q\subset \mathbb{R}^n}\frac{1}{u(Q)}\int_{Q}M(u\chi_Q)(x){\rm d}x,$$
see \cite{wil}.
To consider the weighted estimates with $A_p(\mathbb{R}^n)$ boundedness of singular integral operators with nonsmooth kernels, Martell \cite{mar} introduced the following assumptions.
\begin{assumption}\label{a1.1} There exists an approximation to the identity $\{D_t\}_{t>0}$ such that the composite
operator $D_tT$  has an associated kernel $K^t$ in the sense of (\ref{equ:1.1}), and there exist  positive
constants $c_2$ and $\alpha\in (0,\,1]$,  such that for all $t>0$ and $x,\,y\in\mathbb{R}^n$ with $|x-y|\geq c_2t^{\frac{1}{s}}$,
\begin{eqnarray*}&&|K(x,\,y)-K^t(x,\,y)|\lesssim\frac{t^{\alpha/s}}{|x-y|^{n+\alpha}}.
\end{eqnarray*}
\end{assumption}

\begin{assumption}\label{a1.2} There exists an approximation to the identity $\{A_t\}_{t>0}$ such that the composite
operator $TA_t$  has an associated kernel $K_t$ in the sense of (\ref{equ:1.1}), and there
exists a positive constants $c_1$ and $\alpha\in (0,\,1]$, such that for all $t>0$ with $|x-y|\geq c_1t^{\frac{1}{s}}$,
\begin{eqnarray*}&&|K(x,\,y)-K_t(x,\,y)|\lesssim\frac{t^{\alpha/s}}{|x-y|^{n+\alpha}}.
\end{eqnarray*}
\end{assumption}
Martell \cite{mar}
proved that if $T$ is an $L^2(\mathbb{R}^n)$ bounded linear operator, satisfies Assumption \ref{a1.0} and
Assumption \ref{a1.1}, then for any $p\in (1,\,\infty)$ and $w\in A_p(\mathbb{R}^n)$, $T$ is bounded on $L^p(\mathbb{R}^n,\,w)$,
and if $T$ satisfies Assumption \ref{a1.1} and Assumption \ref{a1.2}, then  for $w\in A_1(\mathbb{R}^n)$, $T$ is bounded from $L^1(\mathbb{R}^n,\,w)$ to $L^{1,\,\infty}(\mathbb{R}^n,\,w)$.

In recent years, there has been significant progress in the study of quantitative weighted bounds for classical operators in harmonic analysis. The study in this field was begun by Buckly \cite{bu} for the Hardy-Littlewood maximal operator, and then by Petermichl \cite{pet1,pet2} for Hilbert transform and Riesz transform. Hyt\"onen solved the so-called $A_2$ conjecture and proved that
for a  Calder\'on-Zygmund operator $T$ and $w\in A_2(\mathbb{R}^n)$,
\begin{eqnarray}\label{equ:A2}\|Tf\|_{L^{2}(\mathbb{R}^n,\,w)}\lesssim_{n}[w]_{A_2}\|f\|_{L^{2}(\mathbb{R}^n,\,w)},\end{eqnarray}
and this weighted bound is sharp.
Hyt\"onen and P\'erez improved estimate (\ref{equ:A2}), and proved that
\begin{eqnarray}\label{equ:1.5}\|Tf\|_{L^{p}(\mathbb{R}^n,\,w)}\lesssim_{n,\,p}[w]_{A_p}^{\frac{1}{p}} \big([w]_{A_{\infty}}^{\frac{1}{p'}}+[\sigma]_{A_{\infty}}^{\frac{1}{p}}\big)\|f\|_{L^{p}(\mathbb{R}^n,\,w)},\end{eqnarray}
where and in the following, for $p\in (1,\,\infty)$ and $w\in A_p(\mathbb{R}^n)$, $p'=p/(p-1)$, $\sigma=w^{-\frac{1}{p-1}}$.
For other  works about the quantitative weighted bounds for Calder\'on-Zygmunds,  see \cite{hlp,hp,hp2,ler2,ler3,ler4} and the related references therein.

Associated with the singular integral operator $T$ in (\ref{equ:1.1}), we define the maximal operator $T^*$ by
$$T^*f(x)=\sup_{\epsilon>0}|T_{\epsilon}f(x)|,$$
with $$T_{\epsilon}f(x)=\int_{|x-y|>\epsilon}K(x,\,y)f(y)dy.
$$
Fairly recently, Hu \cite{hu1} considered the quantitative weighted bounds for singular integral operators with nonsmooth kernels,  and proved the following theorem.
\begin{theorem}\label{thm1.1}
Let $T$ be an $L^2(\mathbb{R}^n)$ bounded linear operator with kernel $K$  in the sense of (\ref{equ:1.1}). Suppose that $T$ satisfies Assumption \ref{a1.1} and   Assumption \ref{a1.2}.
\begin{itemize}
\item[\rm (i)] For $p\in (1,\,\infty)$ and $w\in A_{p}(\mathbb{R}^n)$,
\begin{eqnarray}\label{equa:1.6}&&\|Tf\|_{L^p(\mathbb{R}^n,\,w)}\lesssim [w]_{A_p}^{\frac{1}{p}} \big([w]_{A_{\infty}}^{\frac{1}{p'}}+[\sigma]_{A_{\infty}}^{\frac{1}{p}}\big)[\sigma]_{A_{\infty}}
\big\|f\big\|_{L^p(\mathbb{R}^n,\,w)}.\end{eqnarray}
Moreover, if the kernels $\{K^t\}_{t>0}$ in  Assumption \ref{a1.1}  satisfy that for all $t >0$ and $x,\,y\in\mathbb{R}^n$ with
$|x-y|\leq c_2t^{\frac{1}{s}}$,
\begin{eqnarray}\label{equa:1.size}|K^t(x,\,y)|\lesssim t^{-\frac{n}{s}},\end{eqnarray}
then (\ref{equa:1.6}) holds true for the maximal operator $T^*$.
\item[\rm (ii)]
For $w\in A_1(\mathbb{R}^n)$,
\begin{eqnarray}\label{equa:1.7}\|Tf\|_{L^{1,\,\infty}(\mathbb{R}^n,w)}\lesssim [w]_{A_1}[w]_{A_{\infty}}\log^2({\rm e}+[w]_{A_{\infty}})\|f\|_{L^1(\mathbb{R}^n,w)}.\end{eqnarray}
\item[\rm (iii)] For $w\in A_1(\mathbb{R}^n)$ and $\lambda>0$,\begin{eqnarray}\label{equa:1.11}&&w\big(\{x\in\mathbb{R}^n:\,|Tf(x)|>\lambda\}\big)\\
&&\quad\lesssim_{n,q} [w]_{A_1}\log^2({\rm e}+[w]_{A_{\infty}})\int_{\mathbb{R}^n}
\frac{|f(x)|}{\lambda}\log\Big({\rm e}+\frac{|f(x)|}{\lambda}\Big)w(x)dx.\nonumber\end{eqnarray}
Moreover, if the kernels $\{K^t\}_{t>0}$ in  Assumption \ref{a1.1}  satisfy (\ref{equa:1.size}), then the estimate (\ref{equa:1.11}) also holds for $T^*$.
\end{itemize}
\end{theorem}

The purpose of this paper is to consider the quantitative weighted bounds for the composition of singular integral operators with nonsmooth kernels.
Our main results can be stated as follows.
\begin{theorem}\label{thm1.2}
Let $T_1$, $T_2$ be  $L^2(\mathbb{R}^n)$ bounded singular integral operators with nonsmooth kernels. Suppose that $T_1$, $ T_2$ satisfy Assumption \ref{a1.1} and   Assumption \ref{a1.2}. Then for $p\in (1,\,\infty)$ and
$w\in A_{p}(\mathbb{R}^n)$,
\begin{eqnarray}\label{equ:1.10}&&\|T_1T_2f\|_{L^{p}(\mathbb{R}^n,w)}\lesssim[w]_{A_p}^{\frac{1}{p}} \big([w]_{A_{\infty}}^{\frac{1}{p'}}+[\sigma]_{A_{\infty}}^{\frac{1}{p}}\big)\big([w]_{A_{\infty}}+[\sigma]_{A_{\infty}}\big)[\sigma]_{A_{\infty}}
\|f\|_{L^p(\mathbb{R}^n,w)}.\end{eqnarray}Moreover, if the kernels $\{K^t\}_{t>0}$ in  Assumption \ref{a1.1}  satisfy (\ref{equa:1.size}), then the estimate (\ref{equ:1.10}) also holds for $T^*_1T_2$.\end{theorem}
\begin{theorem}\label{thm1.3}Let $T_1$, $T_2$ be  $L^2(\mathbb{R}^n)$ bounded singular integral operators with nonsmooth kernels. Suppose that $T_1$, $ T_2$ satisfy Assumption \ref{a1.1} and   Assumption \ref{a1.2}.
Then for any $w\in A_1(\mathbb{R}^n)$ and $\lambda>0$,
\begin{eqnarray}\label{equ:1.11}&&w\big(\{x\in\mathbb{R}^n:\,|T_1T_2f(x)|>\lambda\}\big)\\
&&\quad\lesssim_{n}[w]_{A_1}[w]^2_{A_{\infty}}\log^{2}({\rm e}+[w]_{A_{\infty}})\int_{\mathbb{R}^n}\frac{|f(x)|}{\lambda}\log\Big(
{\rm e}+\frac{|f(x)|}{\lambda}\Big)w(x)dx,\nonumber
\end{eqnarray}and
\begin{eqnarray}\label{equ:1.12}&&w\big(\{x\in\mathbb{R}^n:\,|T_1T_2f(x)|>\lambda\}\big)\\
&&\quad\lesssim_{n}[w]_{A_1}[w]_{A_{\infty}}\log^{2}({\rm e}+[w]_{A_{\infty}})\int_{\mathbb{R}^n}\frac{|f(x)|}{\lambda}\log^2\Big(
{\rm e}+\frac{|f(x)|}{\lambda}\Big)w(x)dx.\nonumber\end{eqnarray}
\end{theorem}

We remark that the properties of compositions of singular integral operators were  considered by \cite{bb},
and  many other authors, see \cite{hu3,str,dilu,ober,phst}.

In what follows, $C$ always denotes a
positive constant that is independent of the main parameters
involved but whose value may differ from line to line. We use the
symbol $A\lesssim B$ to denote that there exists a positive constant
$C$ such that $A\le CB$.  Specially, we use $A\lesssim_{n,p} B$ to denote that there exists a positive constant
$C$ depending only on $n,\,p$ such that $A\le CB$. Constant with subscript such as $c_1$,
does not change in different occurrences. For any set $E\subset\mathbb{R}^n$,
$\chi_E$ denotes its characteristic function.  For a cube
$Q\subset\mathbb{R}^n$ and $\lambda\in(0,\,\infty)$, we use $\ell(Q)$ (${\rm diam}Q$) to denote the side length (diamter) of $Q$, and
$\lambda Q$ to denote the cube with the same center as $Q$ and whose
side length is $\lambda$ times that of $Q$. For $x\in\mathbb{R}^n$ and $r>0$, $B(x,\,r)$ denotes the ball centered at $x$ and having radius $r$. For  $\beta\in [0,\,\infty)$,  cube $Q\subset \mathbb{R}^n$ and a suitable function $g$, $\|g\|_{L(\log L)^{\beta},\,Q}$ is the norm defined by
$$\|g\|_{L(\log L)^{\beta},\,Q}=\inf\Big\{\lambda>0:\,\frac{1}{|Q|}\int_{Q}\frac{|g(y)|}{\lambda}\log^{\beta}\Big({\rm e}+\frac{|g(y)|}{\lambda}\Big)dy\leq 1\Big\}.$$
We denote $\|g\|_{L(\log L)^{0},\,Q}$ by $\langle |g|\rangle_{Q}$. For $r\in (0,\,\infty)$, we set $\langle |g|\rangle_{r, Q}=\big(\langle|g|^r\rangle_{Q}\big)^{\frac{1}{r}}.$

\section{Some endpoint estimates}
For $\beta\in [0,\,\infty)$, let $M_{L(\log L)^{\beta}}$ be the maximal operator defined by
 $$M_{L(\log L)^{\beta}}g(x)=\sup_{Q\ni x}\|g\|_{L(\log L)^{\beta},\,Q}.$$
For simplicity, we denote $M_{L(\log L)^{1}}$ by $M_{L\log L}$. Carozza and  Passarelli di Napoli \cite{cana} proved that for $\alpha,\,\beta\in [0,\,\infty)$,
\begin{eqnarray}\label{equ:2.1}
M_{L(\log L)^{\alpha}}\big(M_{L(\log L)^{\beta}}f\big)(x)\approx M_{L(\log L)^{\alpha+\beta+1}}f(x).
\end{eqnarray}
Also, we have that for any $\lambda>0$,
\begin{eqnarray}\label{equ:2.2}\big|\{x\in\mathbb{R}^n:\,M_{L(\log L)^{\beta}}g(x)>\lambda\}\big|\lesssim \int_{\mathbb{R}^n}\frac{|g(x)|}{\lambda}\log^{\beta} \Big({\rm e}+\frac{|g(x)|}{\lambda}\Big)dx.\end{eqnarray}
This section is devoted to the endpoint estimates for the  composite operators $M_{L(\log L)^{\beta}}T_2$ and $M_{L(\log L)^{\beta}}T_1T_2$, with $T_1,\,T_2$   singular integral operators with nonsmooth kernels. These endpoint estimates play important roles in the proof of Theorem \ref{thm1.2} and are of independent interest. To begin with, we give some preliminary lemmas.
\begin{lemma}\label{lem2.1}
Let $p_0\in (1,\,\infty)$, $\beta,\,\varrho\in [0,\,\infty)$ and $S$ be a sublinear operator. Suppose that
$$\|Sf\|_{L^{p_0}(\mathbb{R}^n)}\le A_1 \|f\|_{L^{p_0}(\mathbb{R}^n)},$$
and for all $\lambda>0$,
$$\big|\{x\in\mathbb{R}^n:\,|Sf(x)|>\lambda\}\big|\le A_2 \int_{\mathbb{R}^n}\frac{|f(x)|}{\lambda}
\log ^{\varrho}\Big({\rm e}+\frac{|f(x)|}{\lambda}\Big)dx.$$
Then for two  cubes $Q_2,\,Q_1\subset \mathbb{R}^n$,
\begin{eqnarray*}&&\int_{Q_1}|S(f\chi_{Q_2})(x)|\log^{\beta}\big({\rm e}+|S(f\chi_{Q_2})(x)|\big)dx\\
&&\quad\lesssim |Q_1|+(A_1^{p_0}+A_2)\int_{Q_2}|f(x)|\log^{\beta+\varrho+1} ({\rm e}+|f(x)|) dx.\end{eqnarray*}
\end{lemma}

For the case of $\beta=0$,  Lemma \ref{lem2.1} was essentially proved in the proof of Lemma 3.1 in \cite{huyang}. Fr $\beta\in (0,\,\infty)$, the proof is similar and will be omitted for brevity.

\begin{lemma}\label{lem2.2}
Let $s\in [0,\,\infty)$, $T$ be a sublinear operator which satisfies that for any $\lambda>0$,
$$\big|\{x\in \mathbb{R}^n:\, |Tf(x)|>\lambda\}\big|\lesssim \int_{\mathbb{R}^n}\frac{|f(x)|}{\lambda}\log^s\Big({\rm e}+\frac{|f(x)|}{\lambda}\Big)dx.
$$
Then for any $\varrho\in (0,\,1)$ and cube $Q\subset \mathbb{R}^n$,$$
\Big(\frac{1}{|Q|}\int_{Q}|T(f\chi_{Q})(x)|^{\varrho}dx\Big)^{\frac{1}{\varrho}}\lesssim \|f\|_{L(\log L)^{s},\,Q}.$$
\end{lemma}

For the proof of Lemma \ref{lem2.2}, see \cite[p. 643]{huli}.

\begin{lemma}\label{lem2.3}
Let $R>1$, $\Omega\subset \mathbb{R}^n$ be a open set. Then  $\Omega$ can be decomposed as
$\Omega=\cup_{j}Q_j$, where $\{Q_j\}$ is a sequence of cubes with disjoint interiors, and
\begin{itemize}
\item[\rm (i)]
$$5R\le \frac{{\rm dist}(Q_j,\,\mathbb{R}^n\backslash \Omega)}{{\rm diam} Q_j}\le 15R,$$
\item[\rm (ii)] $\sum_{j}\chi_{RQ_j}(x)\lesssim_{n,R} \chi_{\Omega}(x).$
\end{itemize}
\end{lemma}
For the proof of Lemma \ref{lem2.3}, see \cite[p.\,256]{saw}.

\begin{lemma}\label{lem2.4} Let $\beta\in [0,\,\infty)$,  $U$ be a sublinear operator which is bounded on $L^2(\mathbb{R}^n)$, and satisfies that for any $\lambda>0$,
$$\big|\{x\in\mathbb{R}^n:\, |Uf(x)>t\}\big|\lesssim \int_{\mathbb{R}^n}\frac{|f(x)|}{t}\log ^{\beta}\Big({\rm e}+\frac{|f(x)|}{t}\Big)dx.$$
Let $T$ be an $L^2(\mathbb{R}^n)$ bounded operator with nonsmooth kernel which satisfies  Assumption \ref{a1.2}. Then for any
$\lambda>0$,
\begin{eqnarray}\label{equ:2.3}&&\big|\{x\in\mathbb{R}^n:\, |UTf(x)|>\lambda\}\big|\lesssim
\int_{\mathbb{R}^n}\frac{|f(x)|}{\lambda}\log^{\beta+1} \Big ({\rm e}+\frac{|f(x)|}{\lambda}\Big)dx.\end{eqnarray}
 \end{lemma}
\begin{proof}
By homogeneity, it suffices to prove  inequality (\ref{equ:2.3}) for the case of $\lambda=1$. Without loss of generality, we may assume that $c_1>1$.
Applying Lemma \ref{lem2.3} with $R=3c_1$ to the set $\{x\in\mathbb{R}^n:\,Mf(x)>1\}$, we   obtain a sequence of cubes $\{Q_l\}$ with disjoint interiors, such that
$$\{x\in\mathbb{R}^n:\,Mf(x)>1\}=\cup_lQ_l,$$
and  $$\int_{Q_l}|f(y)|dy\lesssim |Q_l|,\,\,\,\sum_{l}\chi_{3c_1Q_l}\lesssim 1.$$
Let $$g(x)=f(x)\chi_{\mathbb{R}^n\backslash \cup_{l}Q_l}(x),$$
$$b(x)=\sum_{l}f(x)\chi_{Q_l}(x):=\sum_lb_l(x).$$
Recall  that $UT$ is bounded on $L^2(\mathbb{R}^n)$. Thus by the fact that $\|g\|_{L^{\infty}(\mathbb{R}^n)}\lesssim 1$, we get  that
\begin{eqnarray*}\big|\{x\in\mathbb{R}^n:\,|UTg(x)|>\lambda\big\}\big|\lesssim\int_{\mathbb{R}^n}|g(x)|^{2}dx\lesssim \int_{\mathbb{R}^n}|f(x)|dx.\end{eqnarray*}

To estimate $UTb$,   let $t_{Q_l}=\ell(Q_l)^s$ with $s$ the constant appeared in (\ref{equ:1.2}). Write
\begin{eqnarray*}|UTb(x)|&\le & \Big|UT\Big(\sum_{l}A_{t_{Q_l}}b_l\Big)(x)\Big|+\Big|U\Big(\sum_{l}\chi_{3c_1Q_l}Tb_l\Big)(x)\Big|\\
&&+\Big|U\Big(\sum_{l}\chi_{\mathbb{R}^n\backslash 3c_1Q_l}T\big(b_l-A_{t_{Q_l}}b_l\big)\Big)(x)\Big|\\
&&+\Big|U\Big(\sum_{l}\chi_{3c_1Q_l}TA_{t_{Q_l}}b_l\Big)(x)\Big|\nonumber\\
&=&{\rm U}_1(x)+{\rm U}_2(x)+{\rm U}_3(x)+{\rm U}_4(x).\nonumber
\end{eqnarray*}
We first consider the term ${\rm U}_1$.  It was proved in \cite[p. 241]{duongmc} that
\begin{eqnarray}\Big\|\sum_{l}A_{t_{Q_l}}b_l\Big\|_{L^2(\mathbb{R}^n)}^2\lesssim \|f\|_{L^1(\mathbb{R}^n)}.
\end{eqnarray}This, in turn, gives us that
\begin{eqnarray}\label{equ:2.6}\big|\{x\in\mathbb{R}^n: |{\rm U}_1(x)|>1/8\}\big|\lesssim \big\|UT\big(\sum_{l}A_{t_{Q_l}}b_l\big)\big\|_{L^{2}(\mathbb{R}^n)}^{2}\lesssim\|f\|_{L^1(\mathbb{R}^n)}.\end{eqnarray}
Recall that $\chi_{\cup_{l}3c_1Q_l}\lesssim1$. It follows from    Lemma \ref{lem2.1}  that
\begin{eqnarray}&&\big|\{x\in\mathbb{R}^n: |{\rm U}_2(x)|>1/8\}\big|\\
&&\quad\lesssim\sum_l\int_{3c_1Q_l}|Tb_l(x)|\log^{\beta} \Big({\rm e}+\sum_{j}\chi_{3c_1Q_j}|Tb_j(x)| \Big)dx\nonumber\\
&&\quad\lesssim\sum_l\int_{3c_1Q_l}|Tb_l(x)|\log^{\beta} ({\rm e}+|Tb_l(x)| )dx\nonumber\\
&&\quad\lesssim\sum_{l}\Big(|Q_l|+\int_{Q_l}|b_l(y)|\log^{\beta+1} ({\rm e}+|b_l(y)|)dy\Big)\nonumber\\
&&\quad\lesssim\int_{\mathbb{R}^n}|f(y)|\log^{\beta+1} ({\rm e}+|f(y)|)dy.\nonumber\end{eqnarray}

To estimate term ${\rm U}_3$, we first observe that if $v\in L^{2}(\mathbb{R}^n)$ with  $\|v\|_{L^2(\mathbb{R}^n)}=1$, then
\begin{eqnarray*}
&&\sum_{l}\Big|\int_{\mathbb{R}^n\backslash 3c_1Q_l}T\big(b_l-A_{t_{Q_l}}b_l)(y)v(y)dy\Big|\\
&&\quad\lesssim\sum_{l}\int_{\mathbb{R}^n}|b_l(z)|\int_{\mathbb{R}^n\backslash 3c_1Q_l}|K(y,\,z)-K_{t_{Q_l}}(y,\,z)||v(y)|dydz\\
&&\quad\lesssim \sum_{l}\int_{Q_l}Mv(y)dy\lesssim\Big(\sum_l|Q_l|\Big)^{\frac{1}{2}},
\end{eqnarray*}
since
$$\int_{\mathbb{R}^n\backslash 3c_1Q_l}|K(y,\,z)-K_{t_{Q_l}}(y,\,z)||v(y)|dy\lesssim \inf_{z\in Q_l}Mv(z).
$$
Thus, by a standard duality argument,
\begin{eqnarray}\label{equ:2.8}\big|\{x\in\mathbb{R}^n:|{\rm U}_3(x)|>\frac{1}{8}\}\big|&\lesssim&\Big\|\sum_{l}\chi_{\mathbb{R}^n\backslash 3c_1Q_l}T\big(b_l-A_{t_{Q_l}}b_l)\Big\|_{L^2(\mathbb{R}^n)}^2\nonumber\\
&\lesssim&\int_{\mathbb{R}^n}|f(y)|dy.\nonumber\end{eqnarray}

It remains to estimate ${\rm U}_4$. It was proved in \cite[p. 241]{duongmc} that for each fixed $l$,
\begin{eqnarray}\label{equ:2.9}\Big|\int_{\mathbb{R}^n}A_{t_{Q_l}}b_l(x)u(x)dx\Big|\lesssim \int_{Q_l}Mu(x)dx.
\end{eqnarray}Let $v\in L^{2}(\mathbb{R}^n)$ with $\|v\|_{L^{2}(\mathbb{R}^n)}=1$. A straightforward computation involving the inequality (\ref{equ:2.9}) and H\"older's inequality leads to that
\begin{eqnarray*}
\Big|\int_{\mathbb{R}^n}TA_{t_{Q_l}}b_l(x)v(x)dx\Big|
&\lesssim& \int_{Q_l}M\widetilde{T}v(y)dy\\
&\lesssim&|Q_l|^{\frac{1}{2}}\Big(\int_{\mathbb{R}^n}\big[M\widetilde{T}v(x)\big]^{2}dx\Big)^{\frac{1}{2}}
\lesssim|Q_l|^{\frac{1}{2}},
\end{eqnarray*}
here, $\widetilde{T}$ is the adjoint operator of $T$. This, in turn, implies that for each $l$
\begin{eqnarray}\label{equ:2.local}\big\|T_2A_{t_{Q_l}}h_l\big\|_{L^{2}(\mathbb{R}^n)}^{2}\lesssim |Q_l|.\end{eqnarray}
It now follows from the inequality (\ref{equ:2.2}) that
\begin{eqnarray}\label{equ:2.10}\big|\{x\in\mathbb{R}^n: |{\rm U}_4(x)|>1/8\}\big|&\lesssim&
\Big\|\sum_{l}\chi_{3c_1Q_l}TA_{t_{Q_l}}b_l\Big\|_{L^{2}(\mathbb{R}^n)}^{2}\\
&\lesssim&\sum_{l}\int_{3c_1Q_l}|TA_{t_l}b_l(y)|^{2}dy\nonumber\\
&\lesssim& \int_{\mathbb{R}^n}|f(y)|dy.\nonumber\end{eqnarray}
Combining the estimates (\ref{equ:2.6})-(\ref{equ:2.8}) and the estimate (\ref{equ:2.10}) leads to our desired conclusion.
\end{proof}
We are now ready to establish our main result in this section.
\begin{theorem}\label{thm2.1}
Let $T_1$, $T_2$,\,\dots,\,$T_k$ be  $L^2(\mathbb{R}^n)$ bounded operators with nonsmooth kernels. Suppose that $T_1$,\,\dots,\,$T_k$ satisfy Assumption \ref{a1.2}. Then for any
$\lambda>0$,
\begin{eqnarray}\label{equ:2.11}&&\big|\{x\in\mathbb{R}^n:\, |T_1\dots T_kf(x)|>\lambda\}\big|\lesssim \int_{\mathbb{R}^n}\frac{|f(x)|}{\lambda}\log^{k-1} \Big ({\rm e}+\frac{|f(x)|}{\lambda}\Big)dx,\end{eqnarray}and
\begin{eqnarray}\label{equ:2.12}&&\big|\{x\in\mathbb{R}^n:M_{L(\log L)^{\beta}}T_1\dots T_kf(x)>\lambda\}\big|\lesssim \int_{\mathbb{R}^n}\frac{|f(x)|}{\lambda}\log^{\beta+k} \Big ({\rm e}+\frac{|f(x)|}{\lambda}\Big)dx.\end{eqnarray}
\end{theorem}
\begin{proof}We only prove the inequality (\ref{equ:2.12}). For the case of $k=1$, (\ref{equ:2.12}) follows from the inequality (\ref{equ:2.2})
and Lemma \ref{lem2.4}. For general $k\in\mathbb{N}$, (\ref{equ:2.12}) can be deduced by applying Lemma \ref{lem2.4} and the inductive argument.
\end{proof}

\begin{remark}
The inequality (\ref{equ:2.12}) with $\beta=0$ and $k=1$ was proved in \cite{hu1} under the hypothesis that $T$ satisfies
Assumption \ref{a1.1} and Assumption \ref{a1.2}, by a different argument.

\end{remark}

\section{Proof of Theorem \ref{thm1.2}}
Let $T_1$, $T_2$ be two singular integral operators with nonsmooth kernels.
As in Lerner \cite{ler3}, we define the grand maximal operator $\mathcal{M}_{T_1}$  by
$$\mathcal{M}_{T_1}f(x)=\sup_{Q\ni x}{\rm ess}\sup_{\xi\in Q}|T_1(f\chi_{\mathbb{R}^n\backslash 3Q})(\xi)|.$$
It was proved in \cite{hu1} that
\begin{eqnarray}\label{equ:3.-1}\mathcal{M}_{T_1}f(x)\lesssim MT_1f(x)+M_{L\log L}f(x).
\end{eqnarray}
Also, we define the grand maximal operator  $\mathcal{M}^*_{M_{L(
log L)^k}T_1}$ by
$$\mathcal{M}^*_{M_{L(\log L)^k}T_1}f(x)=\sup_{Q\ni x}{\rm ess}\sup_{\xi\in Q}|M_{L(\log L)^k}T_1(f\chi_{\mathbb{R}^n\backslash 9Q})(\xi)|,$$
and  the grand maximal operator $\mathcal{M}_{MT_1T_2}^{**}$ by
$$\mathcal{M}_{MT_1T_2}^{**}f(x)=\sup_{Q\ni x}{\rm ess}\sup_{\xi\in Q}|MT_1\big(\chi_{\mathbb{R}^n\backslash 9Q}
T_2(f\chi_{\mathbb{R}^n\backslash 27Q})\big)(\xi)|.$$
\begin{lemma}\label{lem3.1}
Let $T_1$, $T_2$ be two singular integral operators with nonsmooth kernels. Under the hypothesis of Theorem \ref{thm1.2}, for each bounded function $f$ with compact support,
$$\mathcal{M}_{MT_1T_2}^{**}f(x)\lesssim MT_1T_2f(x)+M_{L\log L}T_2f(x)+M_{L(\log L)^2}f(x).$$
\end{lemma}
\begin{proof}
At first, we claim that for $k\in\mathbb{N}\cup\{0\}$,\begin{eqnarray}\label{equ:3.1}\mathcal{M}_{M_{L(\log L)^k}T_1}^*f(x)\lesssim M_{L(\log L)^k}T_1f(x)+M_{L(\log L)^{k+1}}f(x).\end{eqnarray}
In fact, for each fixed
$Q\subset \mathbb{R}^n$, $x,\,\xi\in Q$, we can write
\begin{eqnarray*}M_{L(\log L)^k}T_1(g\chi_{\mathbb{R}^n\backslash 9Q})(\xi)&\leq& M_{L(\log L)^k}\big(\chi_{\mathbb{R}^n\backslash 3Q}Tg\big )(\xi)\\
&&+M_{L(\log L)^k}\big(\chi_{\mathbb{R}^n\backslash 3Q}T_1(g\chi_{9Q})\big )(\xi)\\
&&+M_{L(\log L)^k}\big(\chi_{3Q}T_1(g\chi_{\mathbb{R}^n\backslash 9Q})\big)(\xi).
\end{eqnarray*}
As it is well known, for each fixed $\xi\in Q$,
$$ M_{L(\log L)^k}\big(\chi_{\mathbb{R}^n\backslash 3Q}T_1g\big )(\xi)\lesssim\inf _{y\in Q}M_{L(\log L)^k}
\big(\chi_{\mathbb{R}^n\backslash 3Q}T_1g\big )(y)\lesssim M_{L(\log L)^k}T_1g(x).
$$
On the other hand, it follows from Lemma \ref{lem2.2} that for $\delta\in (0,\,1)$,
\begin{eqnarray*} M_{L(\log L)^k}\big(\chi_{\mathbb{R}^n\backslash 3Q}T_1(g\chi_{9Q})\big )(\xi)&\lesssim&\inf _{y\in Q}M_{L(\log L)^k}
\big(\chi_{\mathbb{R}^n\backslash 3Q}T_1(g\chi_{9Q})\big )(y)\\
&\lesssim &\Big(\frac{1}{|Q|}\int_{Q}|M_{L(\log L)^k}T_1(g\chi_{9Q})(z)|^{\delta}dz\Big)^{\frac{1}{\delta}}\\
&\lesssim&\|g\|_{L(\log L)^{k+1},\,Q}\lesssim M_{L(\log L)^{k+1}}g(x).
\end{eqnarray*}
The inequality (\ref{equ:3.-1}) tells us that
$$
\sup_{y\in 3Q}|T_1(g\chi_{\mathbb{R}^n\backslash 9Q})(y)|\lesssim \inf_{y\in 3Q}\mathcal{M}_{T_1}g(y)\lesssim  MT_1g(x)+M_{L\log L}g(x).
$$
Therefore, for each $\xi\in Q$,
$$M_{L(\log L)^k}\big(\chi_{3Q}T(g\chi_{\mathbb{R}^n\backslash 9Q})\big)(\xi)\lesssim MT_1g(x)+M_{L\log L}g(x).$$
The estimate (\ref{equ:3.1}) holds true.

We now conclude the proof of Lemma \ref{lem3.1}. Let $x\in \mathbb{R}^n$ and $Q$ be a cube containing $x$.
A trivial computation involving (\ref{equ:3.1}) leads to that for each $\xi\in Q$,
\begin{eqnarray*}
&&MT_1\big(\chi_{\mathbb{R}^n\backslash 9Q}T_2(f\chi_{\mathbb{R}^n\backslash 27Q})\big)(\xi)\lesssim
\inf_{z\in Q}\mathcal{M}^*_{MT_1}(T_2f\chi_{\mathbb{R}^n\backslash 27Q})(z)\\
&&\quad\lesssim\Big(\frac{1}{|Q|}\int_{Q}\big(\mathcal{M}^*_{MT_1}(T_2f\chi_{\mathbb{R}^n\backslash 27Q})(z)\big)^{\frac{1}{2}}dz\Big)^2\\
&&\quad\lesssim\Big(\frac{1}{|Q|}\int_{Q}\big(MT_1T_2f(z)\big)^{\frac{1}{2}}dz\Big)^2
+\Big(\frac{1}{|Q|}\int_{Q}\big(M_{L\log L}T_2f(z)\big)^{\frac{1}{2}}dz\Big)^2\\
&&\qquad+\Big(\frac{1}{|Q|}\int_{Q}\big(MT_1T_2(f\chi_{27Q})(\xi)\big)^{\frac{1}{2}}d\xi\Big)^2\\
&&\qquad+\Big(\frac{1}{|Q|}\int_{Q}\big(M_{L\log L}(T_2f\chi_{27Q})(\xi)\big)^{\frac{1}{2}}d\xi\Big)^2.
\end{eqnarray*}
Recall that $M_{\frac{1}{2}}Mh(x)\lesssim Mh(x)$ and $M_{\frac{1}{2}}M_{L\log L}h(x)\lesssim M_{L\log L}h(x)$. We know that
\begin{eqnarray*}&&\Big(\frac{1}{|Q|}\int_{Q}\big(MT_1T_2f(\xi)\big)^{\frac{1}{2}}d\xi\Big)^2
+\Big(\frac{1}{|Q|}\int_{Q}\big(M_{L\log L}T_2f(\xi)\big)^{\frac{1}{2}}d\xi\Big)^2\\
&&\quad\lesssim MT_1T_2f(x)+M_{L\log L}T_2f(x).
\end{eqnarray*}
On the other hand, it follows from Lemma \ref{lem2.2} that
\begin{eqnarray*}
&&\Big(\frac{1}{|Q|}\int_{Q}\big(MT_1T_2(f\chi_{27Q})(\xi)\big)^{\frac{1}{2}}d\xi\Big)^2+\Big(\frac{1}{|Q|}
\int_{Q}\big(M_{L(\log L)^2}(f\chi_{27Q})(\xi)\big)^{\frac{1}{2}}d\xi\Big)^2\\
&&\quad \lesssim\|f\|_{L(\log L)^2,\,Q}\lesssim M_{L(\log L)^2}f(x).
\end{eqnarray*}
Combining the estimates above yields desired conclusion.
\end{proof}
\begin{remark}\label{rem3.1} Recall that $M_{L(\log L)^k}$ is bounded on $L^p(\mathbb{R}^n)$ for all $p\in (1,\,\infty)$. Thus by (\ref{equ:3.1}), we know from Theorem 2.3 in \cite{hu4} that for each bounded function $f$ with compact support, there exists a sparse family of cubes $\mathcal{S}$, such that for a. e. $x\in\mathbb{R}^n$,
$$M_{L(\log L)^k}T_1f(x)\lesssim \sum_{Q\in \mathcal{S}}\|f\|_{L(\log L)^{k+1},\,Q}\chi_Q(x).$$
\end{remark}
\begin{remark}\label{rem3.2} We define the grand maximal operator $\mathcal{M}_{T_1T_2}^{**}$ by $$\mathcal{M}_{T_1T_2}^{**}f(x)=\sup_{Q\ni x}{\rm ess}\sup_{\xi\in Q}|T_1\big(\chi_{\mathbb{R}^n\backslash 9Q}
T_2(f\chi_{\mathbb{R}^n\backslash 27Q})\big)(\xi)|.$$
It follows from Lemma \ref{lem3.1} that $$\mathcal{M}_{T_1T_2}^{**}f(x)\lesssim MT_1T_2f(x)+M_{L\log L}T_2f(x)+M_{L(\log L)^2}f(x).$$

\end{remark}
As in Lerner \cite{ler4}, we define the bi-sublinear grand maximal operator $\mathscr{M}_{MT_1T_2}^{*}$ by
$$\mathscr{M}_{MT_1T_2}^{*}(f,g)(x)=\sup_{Q\ni x}\frac{1}{|Q|}\int_{Q}MT_1\big(\chi_{9Q}T_2(f\chi_{\mathbb{R}^n\backslash 27Q})\big)(\xi)|g(\xi)|d\xi.
$$
\begin{lemma}\label{lem3.2}
Let $T_1$, $T_2$ be  two singular integral operators with nonsmooth kernels. Under the hypothesis of Theorem \ref{thm1.2}, there exists an operators $U$,  such that
\begin{itemize}
\item[\rm (i)] for any bounded function $f$ with compact support and  any $\lambda>0$,
\begin{eqnarray*}|\{x\in\mathbb{R}^n:\,|Uf(x)|>\lambda\}|\lesssim \int_{\mathbb{R}^n}\frac{|f(x)|}{\lambda}\log\Big({\rm e}+\frac{|f(x)|}{\lambda}\Big)dx;\end{eqnarray*}
\item[\rm (ii)]
for any $q\in (1,\,2]$, bounded function $f$ with compact support and $g\in L^q_{{\rm loc}}(\mathbb{R}^n)$,
\begin{eqnarray*}\mathscr{M}_{MT_{1}T_2}^*(f,g)(x)&\lesssim &q'Uf(x)M_{q}g(x).\end{eqnarray*}
\end{itemize}
\end{lemma}
\begin{proof} Let $Q\subset \mathbb{R}^n$ and $x\in Q$. Recall that for $q\in (1,\,2]$, $M$ is bounded on $L^{q'}(\mathbb{R}^n)$ with bound less than some universal constant $C$, and  $T_1$ is bounded on $L^{q'}(\mathbb{R}^n)$ with bound $Cq'$. It follows from H\"older's inequality that
\begin{eqnarray*}&&\frac{1}{|Q|}\int_{Q}MT_1\Big(\chi_{9Q}T_2(f\chi_{\mathbb{R}^n\backslash 27Q})\Big)(\xi)|g(\xi)|d\xi\\
&&\quad\lesssim
q'\Big(\frac{1}{|9Q|}\int_{9Q}|g(\xi)|^{q}d\xi\Big)^{\frac{1}{q}}
\mathcal{M}_{T_2}f(x)\\
&&\quad\lesssim  q'\big(MT_2f(x)+M_{L\log L}f(x)\big)M_{q}g(x).
\end{eqnarray*}
Let $Uf(x)=MT_2f(x)+M_{L\log L}f(x).$ Our desired conclusion then follows from Theorem \ref{thm2.1}.
\end{proof}
Let $\eta\in (0,\,1)$ and $\mathcal{S}=\{Q_j\}$ be a family of cubes. We say that $\mathcal{S}$ is $\eta$-sparse,  if for each fixed $Q\in \mathcal{S}$, there exists a measurable subset $E_Q\subset Q$, such that $|E_Q|\geq \eta|Q|$ and $E_{Q}$'s are pairwise disjoint.
Associated with  the sparse family $\mathcal{S}$, we define the bi-sublinear sparse operator
$\mathcal{A}_{\mathcal{S};\,L(\log L)^2,\,L^1}$  by
$$\mathcal{A}_{\mathcal{S};\,L(\log L)^2,\,L^1}(f,g)=
\sum_{Q\in\mathcal{S}}|Q|\|f\|_{L(\log L)^2,\,Q}\langle|g|\rangle_{Q},$$
the bi-sublinear
sparse operator $\mathcal{A}_{\mathcal{S},\,L\log L,\,L^q}$ (with $q\in [1,\,\infty)$) by
$$\mathcal{A}_{\mathcal{S};\,L\log L,L^q}(f,g)=\sum_{Q\in\mathcal{S}}|Q|\|f\|_{L\log L,Q}\langle |g|\rangle_{q,\,Q},$$
and the bi-sublinear sparse operator $\mathcal{A}_{\mathcal{S};L^{q_1},L^{q_2}}$ by
$$\mathcal{A}_{\mathcal{S};L^{q_1},L^{q_2}}(f,g)=\sum_{Q\in\mathcal{S}}\langle|f|\rangle_{q_1,\,Q}\langle|g|
\rangle_{q_2,\,Q}|Q|.$$

\begin{theorem}\label{thm3.2}Let $T_1$ and $T_2$ be singular integral operators with nonsmooth kernels. Suppose that both of $T_1$ and $T_2$ satisfy Assumption \ref{a1.1}  and Assumption \ref{a1.2}. Then for each bounded function $f$ with compact support, there exists a $\frac{1}{2}\frac{1}{27^n}$-sparse family of cubes $\mathcal{S}=\{Q\}$, such that for each  function $g\in L^2_{{\rm loc}}(\mathbb{R}^n)$,
$$\Big|\int_{\mathbb{R}^n}g(x)T_1T_2f(x)dx\Big|\lesssim \mathcal{A}_{\mathcal{S};L(\log L)^2,\,L^1}(f,g)+\mathcal{A}_{\mathcal{S};\,L\log L,L\log L}(f,\,g).
$$
\end{theorem}
\begin{proof} We will employ the argument in \cite{ler4}, see suitable variance in \cite{hu3}. For a fixed cube $Q_0$, define the  local analogy of $\mathcal{M}_T$ and $\mathcal{M}^{**}_{T_{1}T_{2}}$ by
$$ \mathcal{M}_{T,\,Q_0}f(x)=\sup_{Q\ni x,\, Q\subset Q_0}{\rm ess}\sup_{\xi\in Q}|T(f\chi_{27Q_0\backslash 27Q})(\xi)|,$$
$$\mathcal{M}^{**}_{T_{1}T_2;\,Q_0}f(x)=\sup_{Q\ni x,\, Q\subset Q_0}{\rm ess}\sup_{\xi\in Q}|T_1\big(\chi_{\mathbb{R}^n\backslash 9Q}
T_2(f\chi_{27Q_0\backslash 27Q})\big)(\xi)|,$$
respectively. Let $E=\cup_{j=1}^3E_j$ with
$$E_1=\big\{x\in Q_0:\, |T_{1}T_2(f\chi_{9Q_0})(x)|>D\|f\|_{L(\log L)^2,\,27Q_0}\big\},$$
$$E_2=\{x\in Q_0:\,\mathcal{M}_{T_2,\,Q_0}f(x)>D\langle |f|\rangle_{27Q_0}\},$$
and
$$E_3=\big\{x\in Q_0:\, \mathcal{M}^{**}_{T_{1}T_2;\,Q_0}f(x)>D\|f\|_{L\log L,27Q_0}\big\},$$
with $D$ a positive constant. It then follows from Theorem \ref{thm2.1}, estimate (\ref{equ:3.-1}) and Remark \ref{rem3.2}  that
$$|E|\le \frac{1}{2^{n+2}}|Q_0|,$$
if we choose $D$ large enough. Now on the cube $Q_0$, we apply the Calder\'on-Zygmund decomposition to $\chi_{E}$ at level $\frac{1}{2^{n+1}}$, and obtain pairwise disjoint cubes $\{P_j\}\subset \mathcal{D}(Q_0)$, such that
$$\frac{1}{2^{n+1}}|P_j|\leq |P_j\cap E|\leq \frac{1}{2}|P_j|$$
and $|E\backslash\cup_jP_j|=0$.  Observe that $\sum_j|P_j|\leq \frac{1}{2}|Q_0|$.
Let
\begin{eqnarray*}
G_0(x)&=&T_{1} T_2(f\chi_{27Q_0})(x)\chi_{Q_0\backslash \cup_{l}P_l}(x)\\
&&+\sum_{l}T_{1}\Big(\chi_{\mathbb{R}^n\backslash 9P_l}T_2(f\chi_{27Q_0\backslash 27P_l}\big)\Big)(x)\chi_{P_l}(x).
\end{eqnarray*}
The facts that $P_j\cap E^c\not =\emptyset$ and $|E\backslash\cup_jP_j|=0$ imply that
\begin{eqnarray}\label{eq3.8}|G_0(x)|\lesssim \|f\|_{L(\log L)^2,\,27Q_0}.\end{eqnarray}
Also, we define function $G_1$ by
$$G_1(x)=\sum_lT_{1}\big(\chi_{9P_l}T_2\big(f\chi_{27Q_0\backslash 27P_l})\big)(x)\chi_{P_l}(x).
$$
Let $\widetilde{T}_1$ be the adjoint operator of $T_1$. For each function $g\in L^2_{{\rm loc}}(\mathbb{R}^n)$, we have by Lemma \ref{lem2.1} that
\begin{eqnarray}\label{eq3.10}
\Big|\int_{\mathbb{R}^n}G_{1}(x)g(x)dx\Big|&\le &\sum_l\int_{9P_l}\big|T_2\big(f\chi_{9Q_0\backslash 9P_l}\big)(x)\widetilde{T}_{1}(g\chi_{P_l})(x)\big|dx\\
&\lesssim&\sum_l\inf_{\xi\in P_l}\mathcal{M}_{T_2,\,Q_0}f(\xi)\int_{9P_l}|\widetilde{T}_{1}(g\chi_{P_l})(x)|dx\nonumber\\
&\lesssim&\langle |f|\rangle_{27Q_0}\|g\|_{L(\log L)^2,\,Q_0}|Q_0|.\nonumber
\end{eqnarray}
Moreover,
$$T_{1} T_2(f\chi_{27Q_0})(x)\chi_{Q_0}(x)=G_0(x)+G_1(x)+\sum_{l}T_{1}T_2(\chi_{27P_l})(x)\chi_{P_l}(x).$$

We now repeat the argument above with $T_{1}T_2(f\chi_{27Q_0})(x)\chi_{Q_0}$ replaced by each $T_{1}T_2(\chi_{27P_l})(x)\chi_{P_l}(x)$, and so on.
Let $Q_{0}^{j_1}=P_{j}$,  and  $\{Q_{0}^{j_1...j_{m-1}j_m}\}_{j_m}$ be the cubes obtained at the $m$-th stage of the decomposition process to the cube $Q_{0}^{j_1...j_{m-1}}$.
For each fixed $j_1\dots,j_m$, define the functions $H_{Q_0}^{j_1\dots j_m}f$, $H_{Q_0,1}^{j_1\dots j_m}f$ by
\begin{eqnarray*}
&&H_{Q_0,0}^{j_1\dots j_m}f(x)=T_{1}\big(\chi_{\mathbb{R}^n\backslash 9Q_0^{j_1\dots j_m}}T_{2}(f\chi_{27Q_0^{j_1\dots j_{m-1}}\backslash 27Q_0^{j_1\dots j_m}})\big)(x)\chi_{Q_0^{j_1\dots j_m}}(x),
\end{eqnarray*}
and
$$H_{Q_0,1}^{j_1\dots j_m}f(x)=T_{1}\Big(\chi_{9Q_{0}^{j_1\dots j_m}}T_2(f\chi_{27Q_{0}^{j_1\dots j_{m-1}}\backslash 27Q_{0}^{j_1\dots j_m}})\big)\Big)(x)\chi_{Q_{0}^{j_1\dots j_m}}(x),
$$ respectively.
Set $\mathcal{F}=\{Q_0\}\cup_{m=1}^{\infty}\cup_{j_1,\dots,j_m}\{Q_{0}^{j_1\dots j_m}\}$. Then $\mathcal{F}\subset \mathcal{D}(Q_0)$ be a $\frac{1}{2}$-sparse  family.  Let
\begin{eqnarray*}&&H_{0,\,Q_0}(x)=T_{1}T_2(f\chi_{27Q_0})\chi_{Q_0\backslash \cup_{j_1}Q_{0}^{j_1}}(x)\\
&&\quad+\sum_{m=1}^{\infty}\sum_{j_1,\dots,j_m}T_{1}T_2(f\chi_{27Q_{0}^{j_1\dots j_m}})\chi_{Q_{0}^{j_1\dots j_m}\backslash \cup_{j_{m+1}}Q_{0}^{j_1\dots j_{m+1}}}(x)\\
&&\quad+\sum_{m=1}^{\infty}\sum_{j_1\dots j_m}H_{Q_0}^{j_1\dots j_m}f(x)\chi_{Q_{0}^{j_1\dots j_m}}(x),
\end{eqnarray*}
Also, we define the functions $H_{1,\,Q_0}$ by
\begin{eqnarray*}H_{1,Q_0}(x)&=&\sum_{m=1}^{\infty}\sum_{j_1\dots j_m}H_{Q_0,1}^{j_1\dots j_m}f(x)\chi_{Q_{0}^{j_1\dots j_m}}(x).
\end{eqnarray*}Then for a. e. $x\in Q_0$,
$$T_{1} T_2(f\chi_{27Q_0})(x)=H_{0,Q_0}f(x)+H_{1,Q_0}(x).
$$
Moreover, as in the inequality (\ref{eq3.8})-(\ref{eq3.10}), the process of producing $\{Q_0^{j_1\dots j_m}\}$ tells us that
$$|H_{0,Q_0}f(x)|\chi_{Q_0}(x)\lesssim \sum_{Q\in\mathcal{F}}\|f\|_{L(\log L)^2,\,27Q}\chi_{Q}(x).$$
For any function $g$, we can verify that
$$\Big|\int_{Q_0}g(x)H_{1,Q_0}(x)dx\Big|\lesssim \sum_{Q\in \mathcal{F}}|Q|\|f\|_{L\log L,\,27Q}\|g\|_{L\log L,\,27Q}.$$

We can now conclude the proof of Theorem \ref{thm3.2}. In fact, as in \cite{ler3}, we decompose $\mathbb{R}^n$ by cubes $R_l$, such that ${\rm supp}f\subset 27R_l$ for each $l$, and $R_l$'s have disjoint interiors.
Then for a. e. $x\in\mathbb{R}^n$,
\begin{eqnarray*}T_{1} T_2f(x)&=&\sum_{l}H_{0,R_l}f(x)+\sum_lH_{1,R_l}f(x):=H_0f(x)+H_1f(x).\end{eqnarray*}
Our desired conclusion then follows directly.
\end{proof}

\begin{theorem}\label{thm3.1}Let $T_1$ and $T_2$ be two  operators with nonsmooth kernels, $q\in (1,\,2]$.
Suppose that $T_1$ and $T_2$ satisfy Assumption \ref{a1.1} and Assumption \ref{a1.2}.
Then for  bounded function $f$ with compact support and $g\in L^{q}_{{\rm loc}}(\mathbb{R}^n)$, there exists a $\frac{1}{2}\frac{1}{27^n}$-sparse
family of cubes $\mathcal{S}=\{Q\}$, such that for a. e. $x\in\mathbb{R}^n$,
\begin{eqnarray*}
\int_{\mathbb{R}^n}|g(x)|MT_{1}T_2f(x)dx\lesssim\mathcal{A}_{\mathcal{S};\,L(\log L)^2, L}(f,g)+q'\mathcal{A}_{\mathcal{S};\,L\log L, L^q}(f,g).
\end{eqnarray*}
\end{theorem}
\begin{proof}  We will employ the argument in \cite{ler4}. For a fixed cube $Q_0$, define the  local analogues of $\mathcal{M}_{MT_{1}T_2}^{**}$ and $\mathscr{M}_{MT_{1}T_2}^*$
 by
\begin{eqnarray*}\mathcal{M}_{MT_1T_2,Q_0}^{**}f(x)=\sup_{x\in Q\subset Q_0}{\rm ess}\sup_{\xi\in Q}MT_1\big(\chi_{\mathbb{R}^n\backslash 9Q}T_2(f\chi_{27Q_0\backslash 27Q})\big)(\xi),\end{eqnarray*}
$$\mathscr{M}_{MT_1T_2,\,Q_0}^{*}(f,g)(x)=\sup_{x\in Q\subset Q_0}\frac{1}{|Q|}\int_{Q}MT_1\big(\chi_{9Q}T_2(\chi_{27Q_0\backslash 27Q})\big)(\xi)|g(\xi)|d\xi
$$respectively.
We claim that for each cube $Q_0\subset \mathbb{R}^n$,
there exist pairwise disjoint cubes $\{P_j\}\subset \mathcal{D}(Q_0)$, such that $\sum_{j}|P_j|\leq \frac{1}{2}|Q_0|$, and for a. e. $x\in Q_0$,
\begin{eqnarray}\label{equ:3.2}\,\,\,
\int_{Q_0}MT_1T_2(f\chi_{27Q_0})(x)|g(x)|dx&\lesssim &\|f\|_{L(\log L)^2,\,27Q_0}\langle |g|\rangle_{Q_0}|Q_0|\\
&&+q'\|f\|_{L\log L,\,27Q_0}\langle |g|\rangle_{q,\,Q_0}|Q_0|\nonumber\\
&&+\sum_j\int_{P_j}MT_{1}T_2(f\chi_{27P_j})(x)|g(x)|dx.\nonumber
\end{eqnarray}
To see this, let $E=E_1\cup E_2\cup E_3$ with
$$E_1=\big\{x\in Q_0:\, MT_1T_2(f\chi_{27Q_0})(x)>D\|f\|_{L(\log L)^2,\,27Q_0}\big\},$$
$$E_2=\big\{x\in Q_0:\, \mathcal{M}^{**}_{MT_1T_2,\,Q_0}f(x)>D\|f\|_{L(\log L)^2,\,27Q_0}\big\},$$
and
$$E_3=\big\{x\in Q_0:\, \mathscr{M}_{MT_{1}T_2,\,Q_0}^*(f,\,g)(x)>Dq'\|f\|_{L\log L,27Q_0}\langle |g|\rangle_{q,Q_0}\big\},$$
with $D$ a positive constant. By Theorem \ref{thm2.1}, Lemma \ref{lem3.1} and Lemma \ref{lem3.2}, we know that
$$|E|\le \frac{1}{2^{n+2}}|Q_0|,$$ if we choose $D$ large enough.
Now on the cube $Q_0$, we apply the Calder\'on-Zygmund decomposition to $\chi_{E}$ at level $\frac{1}{2^{n+1}}$, and obtain pairwise disjoint cubes $\{P_j\}\subset \mathcal{D}(Q_0)$, such that
$$\frac{1}{2^{n+1}}|P_j|\leq |P_j\cap E|\leq \frac{1}{2}|P_j|$$
and $|E\backslash\cup_jP_j|=0$.  Observe that $\sum_j|P_j|\leq \frac{1}{2}|Q_0|$ and $P_j\cap E^c\not =\emptyset.$ Therefore,
$$MT_{1}\Big(\chi_{\mathbb{R}^n\backslash 9P_j}T_2(f\chi_{27Q_0\backslash 27P_j})\Big)(x)\chi_{P_j}(x)\leq D\|f\|_{L(\log L)^2,\,27Q_0},$$ and
$$\frac{1}{|P_j|}\int_{P_j}MT_{1}\Big(\chi_{9P_j}T_2(f\chi_{27Q_0\backslash 27P_j})\Big)(\xi)|g(\xi)|d\xi\leq Dq'\|f\|_{L\log L,\,27Q_0}\langle |g|\rangle_{q,\,Q_0}.$$
The fact  $|E\backslash\cup_jP_j|=0$ implies that for a. e. $x\in Q_0\backslash \cup_{j}P_j$,
$$MT_{1}T_2(f\chi_{9Q_0})(x) \leq D\|f\|_{L(\log L)^2,\,27Q_0}.
$$
Thus for a. e. $x\in Q_0$,
\begin{eqnarray*}
MT_{1}T_2(f\chi_{27Q_0})(x)&\leq&MT_1T_2(f\chi_{27Q_0})(x)\chi_{Q_0\backslash \cup_{j}P_j}(x)\\
&&+\sum_jMT_{1}\Big(\chi_{\mathbb{R}^n\backslash 9P_j}T_2(f\chi_{27Q_0\backslash 27P_j})\Big)(x)\chi_{P_j}(x)\\
&&+\sum_jMT_{1}\Big(\chi_{9P_j}T_2(f\chi_{27Q_0\backslash 27P_j})\Big)(x)\chi_{P_j}(x)\\
&&+\sum_{j}MT_{1}T_2(f\chi_{27P_j})(x)\chi_{P_j}(x)\\
&\le &C\sum_{j}MT_{1}T_2(f\chi_{27P_j})(x)\chi_{P_j}(x)+\|f\|_{L(\log L)^2,27Q_0}\\
&&+\sum_jMT_{1}\Big(\chi_{9P_j}T_2(f\chi_{27Q_0\backslash 27P_j})\Big)(x)\chi_{P_j}(x).
\end{eqnarray*}
The inequality (\ref{equ:3.2}) now follows.

We can now conclude the proof of Theorem  \ref{thm3.1}. In fact, let $\{P_j^0\}=\{Q_0\}$, $\{P_j^1\}=\{P_j\}$ and $\{P_j^k\}$ are the cubes obtained at the $k$-th stage of the iterative process. Set $\mathcal{F}=\cup_{k=0}^{\infty}\cup_{j}\{P_j^k\}$.
Iterating the estimate (\ref{equ:3.2}), we know that there exists a
$\frac{1}{2}$-sparse family $\mathcal{F}\subset \mathcal{D}(Q_0)$, such that
\begin{eqnarray*}
\int_{Q_0}MT_{1}T_2(f\chi_{27Q_0})(x)|g(x)|dx&\lesssim&\sum_{Q\in \mathcal{F}}\|f\|_{L(\log L)^2,\,27Q}\langle|g|\rangle_{Q}|Q|\\
&&+q'\sum_{Q\in\mathcal{F}}\|f\|_{L\log L,\,27Q}\langle|g|\rangle_{q,Q}|Q|.
\end{eqnarray*}
Now apply the ideas in \cite{ler3,ler4}. Let $f$ be a bounded function with compact support. We decompose $\mathbb{R}^n$ as $\mathbb{R}^n=\cup_{l}Q_l$, the cubes $Q_l$'s have disjoint interiors and ${\rm supp}\,f\subset 27Q_l$ for each $l$. For each $l$, we obtain a $\frac{1}{2}$ sparse family $\mathcal{F}_l\subset \mathcal{D}(Q_l)$, such that
\begin{eqnarray*}\int_{Q_l}MT_{1}T_2(f\chi_{27Q_l})(x)|g(x)|dx&\lesssim&\sum_{Q\in \mathcal{F}_l}\|f\|_{L(\log L)^2,\,27Q}\langle|g|\rangle_{Q}|Q|\\
&&+q'\sum_{Q\in\mathcal{F}_l}\|f\|_{L\log L,\,27Q}\langle|g|\rangle_{q,\,Q}|Q|.
\end{eqnarray*}
Summing over the last inequality for $l$ leads to our conclusion.
\end{proof}

\begin{corollary}\label{c3.1}Let $T_1$, $T_2$ be  $L^2(\mathbb{R}^n)$ bounded singular integral operators with nonsmooth kernels.
Suppose that $T_1$, $ T_2$ satisfy Assumption \ref{a1.1} and   Assumption \ref{a1.2}. Then for $p\in (1,\,\infty)$, $\epsilon\in (0,\,1]$
and weight $u$,
\begin{eqnarray*}
\big\|MT_{1}T_2f\|_{L^{p'}\big(\mathbb{R}^n,\,(M_{L(\log L)^{3p-1+\epsilon}}u)^{1-p'}\big)}\lesssim p'^2
\big[p^2(\frac{1}{\epsilon}\big)^{\frac{1}{p'}}\big]^3\|f\|_{L^{p}(\mathbb{R}^n,\,u^{1-p'})}.
\end{eqnarray*}
\end{corollary}
\begin{proof}  By Theorem \ref{thm3.1}, we know that for any $q\in (1,\,2]$, bounded functions $f$, $g$ with compact supports,
there exists a sparse family $\mathcal{S}$ such that
$$\int_{\mathbb{R}^n}MT_{1}T_2f(x)|g(x)|dx\lesssim q'\mathcal{A}_{\mathcal{S},\,L(\log L)^2, L^q}(f,\,g).
$$As in the proof of  proof of Theorem 1.2 in \cite{hp2} (see also the proof of Theorem 1.7 in \cite{lpr}), we can deduce from the last inequality that for $p\in (1,\,\infty)$, $\epsilon\in (0,\,1]$ and weight $u$,
\begin{eqnarray*}
&&\big\|MT_{1}T_2f\|_{L^{p'}(\mathbb{R}^n,\,\big(M_{L(\log L)^{3p-1+\epsilon}}u\big)^{1-p'})}\\
&&\quad\lesssim p'^2\|M_{L(\log L)^2}f\|_{L^{p'}\big(\mathbb{R}^n,\,(M_{L(\log L)^{3p-1+\epsilon/4}}u)^{1-p'}\big)}\\
&&\quad\lesssim p'^2\big[p^2(\frac{1}{\epsilon}\big)^{\frac{1}{p'}}\big]^3\|f\|_{L^{p'}(\mathbb{R}^n,\,u^{1-p'})},
\end{eqnarray*}
since $M_{L(\log L)^2}f\approx MMMf(x)$, see \cite{cana}. This completes the proof of Corollary \ref{c3.1}.
\end{proof}

To prove Theorem \ref{thm1.2}, we will also need the following lemma.
\begin{lemma}\label{lem3.3} Let $\mathcal{S}$ be a sparse family of cubes, $p\in (1,\,\infty)$ and $w\in A_{p}(\mathbb{R}^n)$. Let $\tau_{w}=2^{11+n}[w]_{A_{\infty}}$ and $\tau_{\sigma}=2^{11+n}[\sigma]_{A_{\infty}}$, $\varepsilon_1=\frac{p-1}{2p\tau_{\sigma}+1}$, and $\varepsilon_2=\frac{p'-1}{2p'\tau_w+1}$. Then \begin{eqnarray}\label{equation3.6}\mathcal{A}_{\mathcal{S};L^{1+\varepsilon_1},\,L^{1+\varepsilon_2}}(f,\,g)&\lesssim& [w]_{A_p}^{1/p}([\sigma]_{A_{\infty}}^{1/p}+[w]_{A_{\infty}}^{1/p'})\|f\|_{L^p(\mathbb{R}^n,\,w)}
\|g\|_{L^{p'}(\mathbb{R}^n,\sigma)}.\end{eqnarray}
\end{lemma}

For the proof of Lemma \ref{lem3.3}, see \cite{hu2}.

\begin{lemma}\label{lem3.4}Let  $\beta_1,\,\beta_2\in \mathbb{N}\cup\{0\}$ and  $U$ be a  linear operator. Suppose that for any bounded function $f$ with compact support, there exists a sparse family of cubes $\mathcal{S}$, such that for any function $g\in L^1(\mathbb{R}^n)$,
$$\Big|\int_{\mathbb{R}^n}Uf(x)g(x)dx\Big|\leq \mathcal{A}_{\mathcal{S};\,L(\log L)^{\beta_1},\,L(\log L)^{\beta_2}}(f,\,g).$$
Then for any $\epsilon\in (0,\,1)$ and weight $u$,
\begin{eqnarray*}&&u(\{x\in\mathbb{R}^n:\, |Uf(x)|>\lambda\})\\
&&\quad\lesssim \frac{1}{\epsilon^{1+\beta_1}}
\int_{\mathbb{R}^n}\frac{|f(x)|}{\lambda}\log ^{\beta_1}\Big({\rm e}+\frac{|f(x)|}{\lambda}\Big)M_{L(\log L)^{\beta_2+\epsilon}}u(x)dx.\nonumber\end{eqnarray*}
\end{lemma}
For the proof of Lemma \ref{lem3.4}, see \cite{hu3}.

{\it Proof of Theorem \ref{thm1.2}}.
Recall that for any $\epsilon\in (0,\,1]$ and cube $I\subset \mathbb{R}^n$,
$$\|f\|_{L(\log L)^{\beta},\,I}\lesssim \frac{1}{\epsilon^{\beta}}\langle |f|\rangle_{I,\,1+\epsilon}.$$
It follows from Lemma \ref{lem3.3} that
\begin{eqnarray}\label{equation3.7}&&\mathcal{A}_{\mathcal{S};\,L(\log L)^{2},\,L}(f,\,g)\lesssim \frac{1}{\varepsilon_1^2} \mathcal{A}_{\mathcal{S};L^{1+\varepsilon_1},\,L^{1+\varepsilon_2}}(f,\,g)\\
&&\quad\lesssim [\sigma]_{A_{\infty}}^2 [w]_{A_p}^{1/p}([\sigma]_{A_{\infty}}^{1/p}+[w]_{A_{\infty}}^{1/p'}\big)\|f\|_{L^p(\mathbb{R}^n,\,w)}\|g\|_{L^{p'}(\mathbb{R}^n,\,\sigma)},
\nonumber\end{eqnarray}
and
\begin{eqnarray}\label{equation3.8}&&\frac{1}{\varepsilon_2}\mathcal{A}_{\mathcal{S};\,L\log L,\,L^{1+\varepsilon_2}}(f,\,g)
\lesssim\frac{1}{\varepsilon_1\varepsilon_2}\mathcal{A}_{\mathcal{S};\,L^{1+\epsilon_1},\,L^{1+\varepsilon_2}}(f,\,g) \\
&&\quad\lesssim [w]_{A_{\infty}}[\sigma]_{A_{\infty}} [w]_{A_p}^{1/p}([\sigma]_{A_{\infty}}^{1/p}+[w]_{A_{\infty}}^{1/p'})
\|f\|_{L^p(\mathbb{R}^n,w)}\|g\|_{L^{p'}(\mathbb{R}^n,\sigma)}.\nonumber\end{eqnarray}
On the other hand, we know that $$|T_1T_2f(x)|\leq MT_1T_2f(x),$$
and $$T_1^*T_2f(x)\lesssim MT_1T_2f(x)+MT_2f(x)$$
if the kernels $\{K^t\}_{t>0}$ in  Assumption \ref{a1.1}  satisfy (\ref{equa:1.size}).
The inequality (\ref{equ:1.10}) now follows from Theorem \ref{thm3.1}, Theorem 3.2 and   estimates (\ref{equation3.7})-(\ref{equation3.8}).\qed

\medskip

It was proved in \cite{hu1} that if $T$ is a singular integral operator with nonsmooth kernel,
satisfies Assumption \ref{a1.1} and Assumption \ref{a1.2}, then for $p\in (1,\,\infty)$, $\epsilon\in (0,\,1]$ and weight $u$,
\begin{eqnarray}\label{equation3.9}\|Tf\|_{L^p(\mathbb{R}^n,\,u)}\lesssim p'^2p^2\big(\frac{1}{\epsilon}\big)^{\frac{1}{p'}}\|f\|_{L^p(\mathbb{R}^n,\,M_{L(\log L)^{p-1+\epsilon}}u)}.\end{eqnarray}
and
\begin{eqnarray}\label{equation3.10}
\|Tf\|_{L^{1,\,\infty}(\mathbb{R}^n,\,w)}\lesssim \frac{1}{\epsilon^2}
\|f\|_{L^1(\mathbb{R}^n,\,M_{L(\log L)^{1+\epsilon}}u)}.
\end{eqnarray}
Obviously, (\ref{equation3.9})  implies that for $T_1$, $T_2$ in Theorem \ref{thm1.2},
\begin{eqnarray}\label{equation3.11}\|T_1T_2f\|_{L^p(\mathbb{R}^n,\,u)}\lesssim p'^{4}\big[p^2\big(\frac{1}{\epsilon}\big)^{\frac{1}{p'}}\big]^{2}\|f\|_{L^p(\mathbb{R}^n,\,M_{L(\log L)^{2p-1+\epsilon}}u)}.
\end{eqnarray}

{\it Proof of Theorem \ref{thm1.3}}. Applying Theorem \ref{thm3.2} and Lemma \ref{lem3.4}, we know that for
$\epsilon\in (0,\,1]$ and weight $u$,
\begin{eqnarray}\label{equation3.12}
&&u\big(\{x\in\mathbb{R}^n:\, |T_1T_2f(x)|>\lambda\}\big)\\
&&\quad\lesssim \int_{\mathbb{R}^n}\frac{|f(x)|}{\lambda}\log^2 \Big({\rm e}+\frac{|f(x)|}{\lambda}\Big)M_{L\log L}u(x)dx\nonumber\\
&&\qquad+\frac{1}{\epsilon^2}\int_{\mathbb{R}^n}\frac{|f(x)|}{\lambda}\log \Big({\rm e}+\frac{|f(x)|}{\lambda}\Big)M_{L(\log L)^{1+\epsilon}}u(x)dx\nonumber\\
&&\quad\lesssim \frac{1}{\epsilon^2}\int_{\mathbb{R}^n}\frac{|f(x)|}{\lambda}\log^2 \Big({\rm e}+\frac{|f(x)|}{\lambda}\Big)M_{L(\log L)^{1+\epsilon}}u(x)dx.\nonumber
\end{eqnarray}
This, via the argument used in the proof of Corollary 1.4 in \cite{hp2} (see also the
proof  of Corollay 1.3 in \cite{ler3}) leads to (\ref{equ:1.12}).

 We now prove (\ref{equ:1.11}). It suffices to prove that for any $\epsilon\in (0,\,1]$ and weight $u$,
\begin{eqnarray}\label{equation3.13}
&&u\big(\{x\in\mathbb{R}^n:\, |T_1T_2f(x)|>\lambda\}\big)\\
&&\quad\lesssim \frac{1}{\epsilon^2}\int_{\mathbb{R}^n}\frac{|f(x)|}{\lambda}\log \Big({\rm e}+\frac{|f(x)|}{\lambda}\Big)M_{L(\log L)^{2+\epsilon}}u(x)dx.\nonumber
\end{eqnarray}
By homogeneity, it suffices to prove   inequality (\ref{equation3.13}) for the case of $\lambda=1$.  Again we assume that $c_1>1$.
Applying the Calder\'on-Zygmund decomposition to $|f|$ at l,   we  obtain a sequence of cubes $\{Q_l\}$ with disjoint interiors, such that
$$1<\frac{1}{|Q|}\int_{Q}|f(y)|dy\lesssim 1$$
and $|f(x)|\chi_{\mathbb{R}^n\backslash \cup_{l}Q_l}(x)\lesssim 1$.
Let $g$, $b=\sum_{l}b_l$ and $t_{Q_l}$  be the same as in the proof of Lemma \ref{lem2.4}.
The estimate (\ref{equation3.12}), along with the fact that $\|g\|_{L^{\infty}(\mathbb{R}^n)}\lesssim 1$, tells us that
\begin{eqnarray*}u\big(\{x\in\mathbb{R}^n:\,|T_{1}T_{2}g(x)|>1\big\}\big)&\lesssim&\frac{1}{\epsilon^2}\int_{\mathbb{R}^n}|g(x)|\log^2({\rm e}+|g(x)|)M_{L(\log L)^{1+\epsilon}}u(x)dx\\
&\lesssim &\frac{1}{\epsilon^2}\int_{\mathbb{R}^n}|f(x)|M_{L(\log L)^{1+\epsilon}}u(x)dx.\end{eqnarray*}

To estimate $T_1T_2b$, let $\Omega=\cup_{l}9nc_1Q_l$ and  $\widetilde{u}(y)=u(y)\chi_{\mathbb{R}^n\backslash \Omega}(y).$ We then have that
\begin{eqnarray}\label{equation3.14}u(\Omega)\lesssim \int_{\mathbb{R}^n}|f(y)|Mu(y)dy\end{eqnarray}
and for any $\gamma\in [0,\infty)$,
\begin{eqnarray}\label{equation3.15}\sup_{y\in 3c_1Q_l}M_{L(\log L)^{\gamma}}\widetilde{u}(y)\approx\inf_{z\in 3c_1Q_l}M_{L(\log L)^{\gamma}}\widetilde{u}(z).\end{eqnarray}
Write
\begin{eqnarray*}T_1T_2b(x)&=&T_1T_2\Big(\sum_{l}A_{t_{Q_l}}b_l\Big)(x)+T_1\Big(\sum_{l}\chi_{3c_1Q_l}T_2b_l\Big)(x)\\
&+&T_1\Big(\sum_{l}\chi_{\mathbb{R}^n\backslash 3c_1Q_l}T_2\big(b_l-A_{t_{Q_l}}b_l\big)\Big)(x)-T_1\Big(\sum_{l}\chi_{3c_1Q_l}T_2A_{t_{Q_l}}b_l\Big)(x)\nonumber\\
&=&{\rm W}_1(x)+{\rm W}_2(x)+{\rm W}_3(x)+{\rm W}_4(x).\nonumber
\end{eqnarray*}

We first consider the term ${\rm W}_1$. Let $p_0=1+\epsilon/4$. Estimate (\ref{equ:2.9}), along with Corollary \ref{c3.1},   inequality (\ref{equation3.15}) and a standard duality  argument, leads to that
\begin{eqnarray*}&&\Big\|T_1T_2\Big(\sum_l A_{t_{Q_l}}b_l\Big) \Big\|_{L^{p_0}(\mathbb{R}^n,\,\widetilde{u})}\\
&&\quad\lesssim \sup_{\big\|v\|_{L^{p_0'}(\mathbb{R}^n,\tilde{u}^{1-p_0'})}\leq 1}\big\|M\widetilde{T}_2\widetilde{T}_1v\big\|_{L^{p_0'}(\mathbb{R}^n,\,(M_{L(\log L)^{2+\epsilon}}\widetilde{u})^{1-p_0'})}\\
&&\qquad\times\Big(\int_{\cup_lQ_l}M_{L(\log L)^{2+\epsilon}}\tilde{u}(z)dz\Big)^{\frac{1}{p_0}}\\
&&\quad\lesssim \big(\frac{1}{\epsilon}\big)^2\Big(\int_{\mathbb{R}^n}|f(y)|M_{L(\log L)^{2+\epsilon}}u(y)dy\Big)^{\frac{1}{p_0}},
\end{eqnarray*}
where in the last inequality, we have invoked the fact that
\begin{eqnarray*}\int_{\cup_lQ_l}M_{L(\log L)^{2+\epsilon}}\tilde{u}(z)dz&\lesssim &
\sum_{l}|Q_l|\inf_{y\in Q_l}M_{L(\log L)^{2+\epsilon}}\tilde{u}(y)\\
&\lesssim&\sum_{l}\int_{Q_l}|f(y)|dy \inf_{z\in Q_l}M_{L(\log L)^{2+\epsilon}}\tilde{u}(z).\end{eqnarray*}
Therefore,
\begin{eqnarray}\label{equation3.16}\,\,u\big(\{x\in\mathbb{R}^n\backslash \Omega: |{\rm W}_1(x)|>1/8\}\big)&\lesssim & \big\|T_1T_2\big(\sum_{l}A_{t_{Q_l}}b_l\big)\big\|_{L^{p_0}(\mathbb{R}^n,\,\widetilde{u})}^{p_0}\\
&\lesssim& \frac{1}{\epsilon^2}\int_{\mathbb{R}^n}|f(y)|M_{L(\log L)^{2+\epsilon}}u(y)dy.\nonumber\end{eqnarray}

We turn our attention to   terms ${\rm W}_2$ and ${\rm W}_3$. It follows from Lemma \ref{lem2.1}, estimates (\ref{equation3.10}) and  (\ref{equation3.15})  that
\begin{eqnarray}\label{equation3.17}&&u\big(\{x\in\mathbb{R}^n\backslash \Omega: |{\rm W}_2(x)|>1/8\}\big)\\
&&\quad\lesssim\frac{1}{\epsilon^2}\sum_l\int_{3c_1Q_l}|T_2b_l(x)|M_{L(\log L)^{1+\epsilon}}\widetilde{u}(x)dx\nonumber\\
&&\quad\lesssim\frac{1}{\epsilon^2}\sum_{l}\Big(|Q_l|+\int_{Q_l}|b_l(y)|\log ({\rm e}+|b_l(y)|)dy\Big)\inf_{y\in 3c_1Q_l}M_{L(\log L)^{1+\epsilon}}\widetilde{u}(y)\nonumber\\
&&\quad\lesssim\frac{1}{\epsilon^2}\int_{\mathbb{R}^n}|f(y)|\log ({\rm e}+|f(y)|)M_{L(\log L)^{1+\epsilon}}u(y)dy.\nonumber\end{eqnarray}
To estimate ${\rm W}_3$, we observe that for function $h$,
\begin{eqnarray*}
&&\Big|\int_{\mathbb{R}^n}\sum_{l}\chi_{\mathbb{R}^n\backslash 3c_1Q_l}T_2\big(b_l-A_{t_{Q_l}}b_l\big)(x)h(x)dx\Big|\\
&&\quad\lesssim \sum_{l}\int_{\mathbb{R}^n}\int_{\mathbb{R}^n\backslash 3c_1Q_l}|K(x,\,y)-K_{t_{Q_l}}(x,\,y)||h(x)|dx|b_l(y)|dy\\
&&\quad\lesssim \int_{\cup_lQ_l}Mh(z)dz
\end{eqnarray*}
Let $p_1=1+\epsilon/4$. We then have that
\begin{eqnarray*}
&&\Big\|\sum_{l}\chi_{\mathbb{R}^n\backslash 3c_1Q_l}T_2\big(b_l-A_{t_{Q_l}}b_l\big)\Big\|_{L^{p_1}(\mathbb{R}^n,\,M_{L(\log L)^{\epsilon/2}}\widetilde{u})}\\
&&\quad\lesssim\sup_{\|h\|_{L^{p_1'}(\mathbb{R}^n,\,(M_{L(\log L)^{\epsilon/2}}\widetilde{u})^{1-p_1'})}\leq 1}\|Mh\|_{L^{p_1'}(\mathbb{R}^n,\,(M_{L(\log L)^{1+\epsilon}}\widetilde{u})^{1-p_1'})}\\
&&\qquad\times \Big(\int_{\cup_lQ_l}M_{L(\log L)^{1+\epsilon}}\widetilde{u}(z)dz\Big)^{1/p_1}\\
&&\quad\lesssim \Big(\int_{\mathbb{R}^n}|f(y)|M_{L(\log L)^{1+\epsilon}}u(y)dy\Big)^{1/p_1}.
\end{eqnarray*}
It now follows from  estimate (\ref{equation3.9}) that
\begin{eqnarray}\label{equation3.18}&&u\big(\{x\in\mathbb{R}^n\backslash\Omega:\, |{\rm W}_3(x)|>1/8\}\big)\\
&&\quad\lesssim\frac{1}{\epsilon^2}\Big\|\sum_{l}\chi_{\mathbb{R}^n\backslash 3c_1Q_l}T_2\big(b_l-A_{t_{Q_l}}b_l\big)\Big\|_{L^{p_1}(\mathbb{R}^n,\,M_{L(\log L)^{\epsilon/2}}\widetilde{u})}^{p_1}\nonumber\\
&&\quad\lesssim\frac{1}{\epsilon^2}\int_{\mathbb{R}^n}|f(y)|M_{L(\log L)^{1+\epsilon}}u(y)dy.\nonumber\end{eqnarray}

Finally, by  inequalities (\ref{equation3.11}), (\ref{equation3.15}) and (\ref{equ:2.local}), we deduce that
\begin{eqnarray}\label{equation3.19}&&u\big(\{x\in\mathbb{R}^n\backslash \Omega: |{\rm W}_4(x)|>1/8\}\big)\\
&&\quad\lesssim\frac{1}{\epsilon^2}\sum_{l}\big\|\chi_{3c_1Q_l}T_2A_{t_{Q_l}}b_l\big\|_{L^{1}(\mathbb{R}^n,\,M_{L(\log L)^{1+\epsilon}}\widetilde{u})}\nonumber\\
&&\quad\lesssim\frac{1}{\epsilon^2}\sum_{l}\big\|T_2A_{t_{Q_l}}b_l(y)\big\|_{L^{2}(\mathbb{R}^n)}|Q_l|^{\frac{1}{2}}\inf_{z\in 3c_1Q_l}M_{L(\log L)^{1+\epsilon}}\widetilde{u}(z)\nonumber\\
&&\quad\lesssim\frac{1}{\epsilon^2}\int_{\mathbb{R}^n}|f(y)|M_{L(\log L)^{1+\epsilon}}u(y)dy.\nonumber\end{eqnarray}
Combining the estimates (\ref{equation3.14}), (\ref{equation3.16})-(\ref{equation3.19}) leads to (\ref{equation3.13}).\qed

\begin{remark}\label{rem3.3}Let $T$ be a singular integral operator with nonsmooth kernel, satisfying Assumption \ref{a1.1} and Assumption \ref{a1.2}. Remark \ref{rem3.1} tells us that for each bounded function $f$ with compact support,
there exists a sparse family $\mathcal{S}$, such that
$$MTf(x)\lesssim \sum_{Q\in \mathcal{S}}\|f\|_{L\log L,\,Q}\chi_{Q}(x),
$$
which, implies that, for any $p\in (1,\,\infty)$, $\epsilon\in (0,\,1)$ and weight $u$,
\begin{eqnarray}\label{equation3.20}&&\big\|MTf\big\|_{L^{p'}\big(\mathbb{R}^n,(M_{L(\log L)^{2p-1+\epsilon}} u)^{1-p'}\big)}
\lesssim p'
\Big[p^2\big(\frac{1}{\epsilon}\big)^{\frac{1}{p'}}\Big]^{2}\big\||f|_{q}\|_{L^{p'}(\mathbb{R}^n,u^{1-p'})}.\end{eqnarray}
Let $p_0$,  $\widetilde{u}$, $A_{t_{Q_l}}b_l$  be the same in the proof of Theorem \ref{thm1.3}.
Applying  (\ref{equation3.20}) and  the ideas used in the proof of (\ref{equation3.13}), we can prove that
for each fixed   $\epsilon\in (0,\,1)$ and weight $u$,
\begin{eqnarray*}\big\|T\big(\sum_{l}A_{t_{Q_l}}b_l\big)\big\|_{L^{p_0}(\mathbb{R}^n,\,\widetilde{u})}
&\lesssim&\frac{1}{\epsilon}\Big(\int_{\mathbb{R}^n}|f(y)|_qM_{L(\log L)^{1+\epsilon}}u(y)dy\Big)^{\frac{1}{p_0}},
\end{eqnarray*}
This, along with the argument used in \cite{hu1} leads to that
\begin{eqnarray*}
u\big(\{x\in\mathbb{R}^n:\, |Tf(x)|>1\}\big)\lesssim \frac{1}{\epsilon}\int_{\mathbb{R}^n}|f(x)|M_{L(\log L)^{1+\epsilon}}u(x)dx.
\end{eqnarray*}
Therefore, for $w\in A_1(\mathbb{R}^n)$,
\begin{eqnarray*}\|Tf\|_{L^{1,\,\infty}(\mathbb{R}^n,\,w)}
\lesssim_{n}[w]_{A_1}[w]_{A_{\infty}}\log({\rm e}+[w]_{A_{\infty}})\|f\|_{L^1(\mathbb{R}^n,\,w)}.\end{eqnarray*}
This, improves estimate (\ref{equa:1.7}).
\end{remark}

\end{document}